\documentclass[a4paper,10pt]{article}
	\usepackage[utf8]{inputenc}
	\usepackage{amssymb}
	\usepackage{amsmath}
	\usepackage{amsfonts}
	\usepackage{amsthm}
	\usepackage{stmaryrd}
	\usepackage{mathrsfs}
	\usepackage{algorithm,algorithmic}
	\usepackage{xcolor}
	\usepackage{tkz-euclide}
	
	\usepackage{amsopn}
	
	\usepackage{tikz}
	\usetikzlibrary{matrix,decorations.pathreplacing,calc}
	
	\pgfkeys{tikz/mymatrixenv/.style={decoration=brace,every left delimiter/.style={xshift=3pt},every right delimiter/.style={xshift=-3pt}}}
	\pgfkeys{tikz/mymatrix/.style={matrix of math nodes,left delimiter=[,right delimiter={]},inner sep=2pt,column sep=1em,row sep=0.5em,nodes={inner sep=0pt}}}
	\pgfkeys{tikz/mymatrixbrace/.style={decorate}}

	\catcode`@=11
	\font\tenex=cmex10 
	\newdimen\p@renwd
	\setbox0=\hbox{\tenex B} \p@renwd=\wd0 
	\def\bmat#1{\begingroup \m@th
	   \setbox\z@\vbox{\def\cr{\crcr\noalign{\kern2\p@\global\let\cr\endline}}%
	     \ialign{$##$\hfil\kern2\p@\kern\p@renwd&\thinspace\hfil$##$\hfil
	       &&\quad\hfil$##$\hfil\crcr
	       \omit\strut\hfil\crcr\noalign{\kern-\baselineskip}%
	       #1\crcr\omit\strut\cr}}%
	   \setbox\tw@\vbox{\unvcopy\z@\global\setbox\@ne\lastbox}%
	   \setbox\tw@\hbox{\unhbox\@ne\unskip\global\setbox\@ne\lastbox}%
	   \setbox\tw@\hbox{$\kern\wd\@ne\kern-\p@renwd\left[\kern-\wd\@ne
	     \global\setbox\@ne\vbox{\box\@ne\kern2\p@}%
	     \vcenter{\kern-\ht\@ne\unvbox\z@\kern-\baselineskip}\,\right]$}%
	   \null\;\vbox{\kern\ht\@ne\box\tw@}\endgroup}
	\catcode`@=12
	
	\usepackage{natbib}
	\setcitestyle{square}
	\setcitestyle{numbers}
	\setcitestyle{comma}
	
	\newcommand{\R}{\mathbb {R}}
	\newcommand{\tridiag}{\ensuremath{\,\mathrm{tridiag}}}
	\newtheorem{thm}{Theorem}
	\newtheorem{cor}[thm]{Corollary}
	\newtheorem{lemma}[thm]{Lemma}
	\newtheorem{remark}[thm]{Remark}
	\newtheorem{defn}[thm]{Definition}
	
	\DeclareMathOperator*{\argmax}{arg\,max}
	
	\title{On maximum volume submatrices and cross approximation for symmetric semidefinite and diagonally dominant matrices}
	
	\author{Alice Cortinovis\footnote{MATH-ANCHP, École Polytechnique Fédérale de Lausanne, Station 8, 1015 Lausanne, Switzerland. E-mail: alice.cortinovis@epfl.ch. The work of Alice Cortinovis has been supported by the SNSF research project \emph{Fast algorithms from low-rank updates}, grant number: 200020\_178806.} \and Daniel Kressner\footnote{MATH-ANCHP, École Polytechnique Fédérale de Lausanne, Station 8, 1015 Lausanne, Switzerland. E-mail: daniel.kressner@epfl.ch} \and Stefano Massei\footnote{MATH-ANCHP, École Polytechnique Fédérale de Lausanne, Station 8, 1015 Lausanne, Switzerland. E-mail: stefano.massei@epfl.ch} }
	
	\date{}
	
\begin{document}
	
	\maketitle
	
	\begin{abstract}
	    The problem of finding a $k \times k$ submatrix of maximum volume of a matrix $A$ is of interest in a variety of applications. For example, it yields a quasi-best low-rank approximation constructed from the rows and columns of $A$. We show that such a submatrix can always be chosen to be a principal submatrix if $A$ is symmetric semidefinite or diagonally dominant. Then we analyze the low-rank approximation error returned by a greedy method for volume maximization, cross approximation with complete pivoting. Our bound for general matrices extends an existing result for symmetric semidefinite matrices and yields new error estimates for diagonally dominant matrices. In particular, for doubly diagonally dominant matrices the error is shown to remain within a modest factor of the best approximation error. We also illustrate how the application of our results to cross approximation for functions leads to new and better convergence results.
	\end{abstract}

	\section{Introduction}
	
	Given a matrix $A \in \R^{n\times n}$ and $1\le k\le n$, the volume of a submatrix $A(I,J)$ for two index sets $I,J\subset \{1,\ldots,n\}$ of cardinality $k$ is defined as the absolute value of the determinant. The problem of finding the submatrix of maximum volume is connected to a range of applications in discrete mathematics, engineering, and scientific computing; see, e.g., \cite{ArioliDuff2015,Gritzmann1995,Wang2010}.  Our primary motivation is its connection to low-rank approximation. Specifically, if $A(I,J)$ is invertible then the matrix $A(:,J) A(I,J)^{-1} A(I,:)$ has rank $k$ where the colon is used to denote that all rows or columns are selected. Low-rank approximations that involve columns and rows of the original matrix come in various flavors, as pseudoskeleton approximation~\cite{GoreinovTyrtyshnikovZamarashkin1997}, cross approximation~\cite{Bebendorf2000}, CUR approximation~\cite{DrineasMahoneyMuthukrishnan2008}, or (strong) rank-revealing LU factorizations~\cite{Miranian2003,Pan2000}.
	If $A(I,J)$ has maximum volume then, by a result of Goreinov and Tyrtyshnikov~\cite{GoreinovTyrtyshnikov2001}, we have
	\begin{equation} \label{eq:goreinov}
	 \|A- A(:,J) A(I,J)^{-1} A(I,:)\|_{\max} \le (k+1) \sigma_{k+1}(A),
	\end{equation}
	where $\|\cdot\|_{\max}$ denotes the maximum absolute value of the entries of a matrix. The $(k+1)$th largest singular value of $A$, denoted by $\sigma_{k+1}(A)$, is well known to govern the best rank-$k$ approximation error of $A$ in the spectral norm; see, e.g., \cite[Chapter 7.4]{HornJohnson2013}. In other words, the result~\eqref{eq:goreinov} states that volume maximization yields a quasi-best low-rank approximation.
	
	Finding the submatrix of maximum volume is a difficult problem. In fact, it is NP hard to determine the maximum volume submatrix of $A$~\cite{CivrilMagdon-Ismail2009,Papadimitriou1984}. 
	In~\cite{DiSumma2015} it is shown that there exists a universal constant $c > 1$ such that it is NP hard to approximate the maximum volume of a $k \times k$ submatrix of a matrix $A \in \R^{n \times k}$ within a factor $c^k$.
	By a trivial embedding, this implies that it is also NP hard to approximate the maximum volume of a $k \times k$ submatrix of an $n \times n$ matrix. However, it is important to emphasize that while a good approximation of the maximum volume submatrix yields a good low-rank approximation~\cite[Theorem 2.2]{GoreinovTyrtyshnikov2001}, the converse is generally not true, see also Remark~\ref{remark:polyapprox} below.
	
	Despite the difficulties associated with volume maximization, this concept has proven fruitful in the development of greedy and randomized algorithms that often yield reasonably good low-rank approximations. In particular, \emph{(adaptive) cross approximation}~\cite{Bebendorf2000}, a greedy method for volume maximization,
	can now be regarded as the work horse for matrices $A$ that cannot be stored in memory, because, for example, $A$ has too many nonzero entries or it is too expensive to compute all the entries of $A$. This situation occurs frequently for discretized integral operators and cross approximation plays an important role in accelerating computations within the boundary element method~\cite{Bebendorf2000} and uncertainty quantification~\cite{HarbrechtPetersSchneider2012}.
	Low-rank approximations of the form~\eqref{eq:goreinov} also feature prominently in the Nyström method for kernel-based learning~\cite{BachJordan2005} 
	 and spectral clustering~\cite{Fowlkes2004}. Let us also stress that cross approximation is equivalent to Gaussian elimination and (incomplete) LU factorization with complete pivoting; it primarily constitutes a different point of view with stronger emphasis on low-rank approximation.
	 
	In many of the applications mentioned above, the matrix $A$ carries additional structure. For example, if $A$ is the discretization of an integral operator with a positive semidefinite kernel then $A$ is symmetric positive semidefinite.
	In the first part of this work, we will show that the submatrix of maximum volume is always attained by a principal submatrix if $A$ is symmetric positive semidefinite (SPSD). This has a number of important consequences. For example, it allows us to draw a one-to-one correspondence to the column selection problem considered in \cite{CivrilMagdon-Ismail2009}. In turn, the maximum volume problem remains NP hard when restricted to SPSD matrices. 
	We also extend this result to diagonally dominant (DD) matrices. Somewhat surprisingly, we have not found such results for SPSD and DD matrices in the existing literature.
	
	In the second part of this work, we derive a priori error bounds for the approximation returned by cross approximation. Although the literature on rank-revealing LU factorizations contains related results, see in particular~\cite[Corollary 5.3]{Foster2006}, the non-asymptotic bound of Theorem~\ref{thm:ACAgeneral} appears to be new.
	Our result includes existing work by Harbrecht et al.~\cite{HarbrechtPetersSchneider2012} for SPSD matrices as a special case. For the particular case of doubly DD matrices, we show that the approximation error returned by cross approximation is at most $2(k+1)$ times larger than the right-hand side of~\eqref{eq:goreinov}. This class of matrices includes symmetric DD matrices, which play a prominent role in \cite{Koutis2016,Spielman2014}. Our result also allows us to obtain refined bounds for the convergence of cross approximation applied to functions~\cite{Bebendorf2000, TownsendTrefethen2015}.

	\section{Maximum volume submatrices} \label{sec:maxvol}
	
	In this section, we will prove for two classes of matrices that the submatrix of maximum volume can always be chosen to be a principal submatrix, that is, a submatrix of the form $A(I,I)$.
	
	\subsection{Symmetric positive semidefinite (SPSD) matrices}

	\begin{thm} \label{thm:spsd}
	  Let $A \in \R^{n\times n}$ be SPSD and let $1 \le k \le n$. Then the maximum volume $k\times k$ submatrix of $A$ can be chosen to be a principal submatrix. 
	\end{thm}
	
	\begin{proof}
	  Let $A(I,J)$ be any $k\times k$ submatrix of $A$. 
	    As $A$ is SPSD, it admits a Cholesky decomposition $A = C^* C$ and, in turn, $A(I,J) = C(:,I)^* C(:,J)$.
The singular values of a principal submatrix satisfy
        \begin{equation*}
            \sigma_i(A(I,I)) = \sigma_i\left (C(:,I)^* C(:,I) \right ) = \sigma_i(C(:,I))^2.
        \end{equation*}
Noting that the absolute value of the determinant equals the product of the singular values, we obtain
        \begin{align*}
	        \det (A(I,J))^2 & = \left ( \Pi_{i=1}^k \sigma_i(A(I,J)) \right )^2 = \left (\Pi_{i = 1}^k \sigma_i (C(:,I)^* C(:,J)) \right )^2 \\
	        & \le \left ( \Pi_{i=1}^k \sigma_i(C(:,I))  \Pi_{j = 1}^k \sigma_j(C(:,J)) \right ) ^2 \\
	        & = \left (  \Pi_{i=1}^k \sigma_i(C(:,I)^* C(:,I))\right) \left ( \Pi_{j=1}^k \sigma_j(C(:,J)^* C(:,J)) \right ) \\ 
	        & = \left ( \Pi_{i = 1}^k \sigma_i(A(I,I))\right)  \left ( \Pi_{j = 1}^k \sigma_j(A(J,J)) \right ) \\
	        & =  \det(A(I,I))\cdot \det(A(J,J)),
        \end{align*}
        where we used~\cite[Theorem 3.3.4]{HornJohnson1991} for the inequality.
	This implies that the volume of $A(I,J)$ is not larger than the maximum of the volumes of $A(I,I)$ and $A(J,J)$. In turn, $A(I,J)$ can be replaced by a principal submatrix without decreasing the volume.
	\end{proof}
	
	Trivially, the result of Theorem~\ref{thm:spsd} extends to symmetric negative semidefinite matrices. On the other hand, it does not extend to the indefinite case; consider for example the $2k\times 2k$ matrix $A = \begin{bmatrix} 0 & I \\ I & 0 \end{bmatrix}$.
	
	\subsubsection{Connection to column selection}
	
	The \emph{volume} of a general $n \times k$ matrix is defined as the product of its singular values.
	The \emph{column selection problem}, which is also connected to low-rank approximation~\cite{Deshpande2006}, is the following: 
	\begin{quote}
	 Given $B \in \R^{n\times m}$ and $1 \le k < m$, select the $n \times k$ submatrix of maximum volume.
	\end{quote}
	 In \cite{CivrilMagdon-Ismail2009} it is shown that this problem is NP hard.
	 
	Theorem \ref{thm:spsd} allows us to relate the column selection problem to the classical maximum volume submatrix problem. 
	Given $B \in \R^{n \times m}$, we consider the SPSD matrix $A = B^* B \in \R^{m \times m}$. As $A(I,I) = B(:,I)^*  B(:,I)$ for any index set $I$, there is a one-to-one correspondence between the principal submatrices of $A$ and the subsets of $k$ columns of $B$.
	Moreover, as seen in the proof of Theorem~\ref{thm:spsd}, the volume of $A(I,I)$ is the square of the volume of $B(:,I)$. This shows that  $B(:,I)$ has maximum volume if and only if $A(I,I)$ has maximum volume. In turn, this proves that the maximum volume submatrix problem remains NP hard when restricted to the subclass of SPSD matrices.

	\subsection{Diagonally dominant (DD) matrices}
	\begin{defn}
	  A matrix $A \in \R^{n\times n}$ is called (row) \emph{diagonally dominant} (DD) if 
	  \begin{equation}\label{eq:DD}
	    \sum_{j = 1, j \neq i}^n |a_{ij}| \le |a_{ii}|, \qquad i = 1, \ldots, n.
	  \end{equation}
	  If~\eqref{eq:DD} holds with strict inequality for $i = 1, \ldots, n$, we call $A$ \emph{strictly DD}.
	  A matrix $A$ is called \emph{doubly DD} if both $A$ and $A^*$ are DD. 
	\end{defn}

	\begin{lemma}\label{lem:WorksForTriang}
	  Let $T \in \R^{n\times n}$ be a strictly DD, upper triangular matrix.  Then $|\det(T(I,J))| < |\det(T(I,I))|$ holds for every 
	  $I,J \subseteq \{1,\ldots,n\}$ with $|I| = |J|$ and $I\not=J$.
	\end{lemma}
	
	\begin{proof}
	Let $D$ be the diagonal matrix with $d_{ii} = t_{ii}$ and set $\tilde T = D^{-1} T$. Because of
	\begin{eqnarray*}
	    \det(T(I,J)) &=& \det(D(I,I))\cdot \det(\tilde T(I,J)), \\
	    \det(T(I,I)) &=& \det(D(I,I))\cdot \det(\tilde T(I,I)),
	\end{eqnarray*}
	the statement of the lemma holds for $T$ if and only if it holds for $\tilde T$. In turn, this allows us to assume without loss of generality that $T$ has ones on the diagonal.
	In particular, $\det(T(I,I)) = 1$.
	
	The statement of the lemma will be proven by induction on $k := |I| = |J|$.
	The case $k = 1$ follows immediately from the diagonal dominance of $T$.
	Suppose now that the statement of the lemma is true for fixed $k$. To prove the statement for $k+1$, we consider an arbitrary $(k+1) \times (k+1)$ submatrix $B := T(I,J)$. If $I \not= J$ then there exists a row $B(i,:)$ that does not contain a diagonal element of $T$. By diagonal dominance of $T$, 
	  \begin{equation} \label{eq:sumsmaller1}
	    |b_{i,1}| + |b_{i,2}| + \ldots + |b_{i,k+1}| < 1.
	\end{equation}
	Denote by $B_{ij}$ the $k\times k$ submatrix of $T$ obtained from eliminating the $i$th row and $j$th column of $B$. By induction assumption, $|\det(B_{ij})|\le 1$. Thus, combining~\eqref{eq:sumsmaller1} with the Laplace expansion gives
	\[
	    |\det(B)|  =  \Big| \sum_{j = 1}^{k+1} (-1)^{i+j} b_{ij} \det(B_{ij}) \Big| \le \sum_{j = 1}^{k+1} |b_{ij}|\, |\det(B_{ij})| 
	    \le \sum_{j = 1}^{k+1} |b_{ij}| < 1.
	\]
	In other words, $|\det(T(I,J))| < |\det(T(I,I))|$.
	\end{proof}
	
	\begin{thm}\label{thm:maxvolDD}
	  Let $A \in \R^{n\times n}$ be a diagonally dominant matrix and $1 \le k \le n$.
	  Then the maximum volume $k\times k$ submatrix of $A$ can be chosen to be a principal submatrix.
	\end{thm}
	
	\begin{proof}
	We prove the theorem in the case when $A$ is strictly DD; the DD case follows by a continuity argument, noting that volumes of submatrices are continuous in $A$. Let $A(I,J)$ be a $k \times k$ submatrix of $A$. Also, by applying a suitable permutation to the rows and columns of $A$, we may assume that $I = \{1,\ldots,k\}$ and $J = \{k-d+1,\ldots,2k-d\}$ with $d = |I\cap J|$. The result of the theorem follows if we can prove 
	  \begin{equation}\label{eq:DetIneqDD}
	    |\det(A(I,J))| \le |\det(A(I,I))|.
	  \end{equation}
	  For this purpose, we note that the LU factorization $A = LU$ always exists with $U$ strictly DD; see Theorem 9.9 in \cite{Higham2002}. We have that
	\[
	    A(I,I) = L(I,I) U(I,I),\quad     A(I,J) = L(I,I) U(I,J).
	  \]
	  As $L(I,I)$ is lower triangular with ones on the diagonal, we obtain
	\[    |\det(A(I,I))| = |\det(U(I,I))|,\quad     |\det(A(I,J)| = |\det(U(I,J))|.
	\]
	Thus, the inequality~\eqref{eq:DetIneqDD} follows from Lemma \ref{lem:WorksForTriang}.
	\end{proof}
	
	For $k = n-1$, the result of Theorem~\ref{thm:maxvolDD}  is covered in the proof of Theorem 2.5.12 in \cite{HornJohnson1991}, while the result of Lemma~\ref{lem:WorksForTriang} for $k = n-1$  follows from Proposition 2.1 in~\cite{Pena2004}.

	\section{Cross Approximation}
	
	In the following, we summarize the idea behind Bebendorf's cross approximation algorithm~\cite{Bebendorf2000}. For this purpose, we first recall 
	that an approximation of the form $A(:,J)A(I,J)^{-1} A(I,:)$ is closely connected to an incomplete LU decomposition of $A$. To see this, suppose that $A$ has been permuted such that $I = J = \{1,\ldots,k\}$ and partition
	\[
	 A = \begin{bmatrix}
	      A_{11} & A_{12} \\
	      A_{21} & A_{22}
	     \end{bmatrix}, \qquad A_{11} \in \R^{k\times k}.
	\]
	Assume that $A_{11}$ is invertible and admits an LU decomposition $A_{11} = L_{11} U_{11}$, where $L_{11}$ is lower triangular and $U_{11}$ is upper triangular with ones on the diagonal. By setting $L_{21} = A_{21} U_{11}^{-1}$ and $U_{12} = L_{11}^{-1} A_{12}$, we obtain
	\begin{eqnarray}
	  A &=& A(:,J)A(I,J)^{-1}A(I,:) + \begin{bmatrix}
	     0 & 0 \\
	     0 & A^{(k)}
	     \end{bmatrix} \nonumber \\  
	     &=& \begin{bmatrix}
	      L_{11}  \\
	      L_{21} 
	     \end{bmatrix}
	     \begin{bmatrix}
	      U_{11} & 
	      U_{12} 
	     \end{bmatrix} + \begin{bmatrix}
	     0 & 0 \\
	     0 & A^{(k)}
	     \end{bmatrix}  \label{eq:sumlu} \\
	     &=& \begin{bmatrix}
	      L_{11} & 0  \\
	      L_{21} & I
	     \end{bmatrix} \begin{bmatrix}
	     I & 0 \\
	     0 & A^{(k)}
	     \end{bmatrix}
	     \begin{bmatrix}
	      U_{11} & 
	      U_{12}  \\
	      0 & I
	     \end{bmatrix}, \label{eq:decomplu}
	\end{eqnarray}
	with the Schur complement \[
	A^{(k)} := A_{22} - A_{21} A_{11}^{-1} A_{12}.                           
	                          \]
	This shows that the approximation error is governed by $A^{(k)}$. The factorized form~\eqref{eq:sumlu} corresponds exactly to what is obtained after applying $k$ steps of the LU factorization to $A$, see, e.g.,~\cite[Chapter 3.2]{GolubVanLoan2013}.
	
	Given index sets $I$ and $J$, one step of the greedy method for volume maximization consists of choosing indices such that
	\begin{equation} \label{eq:localopt}
	 (i_{k+1}, j_{k+1}) = \argmax\big\{\big|\det\!\big(A( I \cup \{i\}, J\cup\{j\} )\big) \big| :\,i\not\in I, j\not\in J\big\}.
	\end{equation}
	Again, let us assume that $I = J = \{1,\ldots,k\}$ and set $\tilde I = I \cup \{k+\tilde i\}$, $\tilde J = J \cup \{k+\tilde j\}$ for some $\tilde i,\tilde j \in \{1,\ldots,n-k\}$. Then~\eqref{eq:decomplu} implies
	\[
	 \det\big(A(\tilde I, \tilde J)\big) = 
	 \det\big(A(I, J)\big) \cdot A^{(k)}(\tilde i,\tilde j).
	\]
	In other words, the local optimization problem~\eqref{eq:localopt} is solved by searching the entry of $A^{(k)}$ that has maximum modulus. This choice leads to Algorithm~\ref{alg:aca}, which is equivalent to applying LU factorization with complete pivoting to $A$.
	\begin{algorithm}
	\caption{Cross approximation with complete pivoting~\cite{Bebendorf2000} \label{alg:aca}}
	 \begin{algorithmic}[1]
	\STATE{Initialize $R_0:=A,\, I:=\{\},\, J:=\{\}.$}
	\FOR{$k = 0,\ldots,m-1$}
	    \STATE{$(i_{k+1},j_{k+1}) := \argmax_{i,j}|R_{k}(i,j)|$} \label{ln:ACAfullpivot}
	    \STATE{$I \gets I\cup\{i_k\},\, J\gets J\cup\{j_k\}$}
	    \STATE{$p_{k+1} := R_{k}(i_{k+1},j_{k+1})$}
	    \STATE{$R_{k+1} := R_{k} - \frac{1}{p_{k+1}} R_{k}(:,j_{k+1}) R_{k}(i_{k+1},:)$} 
	\ENDFOR
	\end{algorithmic}
	\end{algorithm}
	\begin{remark}
	Because of~\eqref{eq:sumlu}, the remainder term of Algorithm~\ref{alg:aca} satisfies $R_{k} = \begin{bmatrix}
	     0 & 0 \\
	     0 & A^{(k)}
	\end{bmatrix}$ after a suitable permutation of the indices. 
	Both for SPSD and DD matrices, the element of maximum modulus is on the diagonal. Positive definiteness and diagonal dominance are preserved by taking Schur complements; see, e.g.,~\cite[Chapter 4]{Zhang2005}. In turn, the search for the pivot element in Step~\ref{ln:ACAfullpivot} can be restricted to the diagonal for such matrices. This  significantly reduces the number of entries of $A$ that need to be evaluated when running Algorithm~\ref{alg:aca}. It also implies that Algorithm~\ref{alg:aca} returns $I = J$, which aligns nicely with the results from Section~\ref{sec:maxvol}.
	Notice that if $A$ is an SPSD matrix then the cross approximation 
	\begin{equation*} 
	 A(:,I)A(I,I)^{-1} A(I,:)
	\end{equation*}
	obtained by Algorithm~\ref{alg:aca} is SPSD. In contrast, diagonal dominance is generally not preserved by the low-rank approximation returned by Algorithm~\ref{alg:aca}.
	\end{remark}

	\subsection{Error analysis for general matrices}
	
	Although not desirable, it may happen that the pivots $p_k$ in Algorithm~\ref{alg:aca} grow. Upper bounds on the \emph{growth factor}
	$
	 \Vert A^{(k)} \Vert_{\max} / \Vert A \Vert_{\max}
	$
	play an important role in the error analysis of Gaussian elimination (see e.g.~\cite{Wilkinson1961}). In the setting of complete pivoting, we can define 
	\begin{equation} \label{eq:growthfactor}
	 \rho_k := 
	 \sup_{A} \big\{ \Vert A^{(k)} \Vert_{\max} / \Vert A \Vert_{\max} \big\},
	\end{equation}
	where the supremum is taken over all matrices of rank at least $k$. This condition ensures that there is no breakdown in the first $k$ steps of Algorithm~\ref{alg:aca}.
	By definition, $1\le \rho_1 \le \rho_2 \le \ldots \le \rho_k$.
	Wilkinson~\cite{Wilkinson1961} proved that
	\begin{equation*}
	    \rho_k \le \sqrt{k+1} \cdot \sqrt{2 \cdot 3^{1/2} \cdot 4^{1/3} \cdot \ldots \cdot (k+1)^{1/k}} \le 2\sqrt{k+1} (k+1)^{\ln (k+1)/4}.
	\end{equation*}
	but it is known that this bound cannot be attained for $k \ge 3$. For matrices occurring in practice, it is rare to see any significant growth and it is not unreasonable to consider $\rho_k = O(1)$; we refer to~\cite[Section 9.4]{Higham2002} for more details. Extending the proof of~\cite[Theorem 3.2]{HarbrechtPetersSchneider2012}, we obtain the following result.
	\begin{thm}\label{thm:ACAgeneral}
	Let $A\in \R^{n\times n}$ have rank at least $m < n$. Then
	the index sets $I,J$ returned by Algorithm~\ref{alg:aca} satisfy
	  \begin{equation}\label{eq:aca-general}
	    \Vert A - A(:,J)A(I,J)^{-1} A(I,:) \Vert_{\max} \le 4^m \cdot \rho_m \cdot \sigma_{m+1}(A).
	  \end{equation}
	\end{thm}
	
	\begin{proof}
	Without loss of generality, we may assume $I = J = \{1,\ldots,m\}$. We perform one more step of Algorithm~\ref{alg:aca} and consider the relation $A_{11} = L_{11} U_{11}$ from~\eqref{eq:sumlu} for $k = m+1$. Because of complete pivoting, the element of largest modulus in the $j$th column of $L_{11}$ is on the diagonal and equals $p_j$.
	For such triangular matrices, Theorem 6.1 in \cite{Higham1987} gives
	\begin{equation*}
	    \Vert L_{11}^{-1} \Vert\le 2^m \cdot \min\{|p_1|,\ldots,|p_{m+1}|\}^{-1},
	\end{equation*}
	where $\Vert\cdot\Vert$ denotes the spectral norm of a matrix.
	Analogously, using that the element of largest modulus in every row of $U_{11}$ is on the diagonal and equals $1$, we obtain $\Vert U_{11}^{-1} \Vert \le 2^m$.
	Hence,
	\[
	 \| A_{11}^{-1} \| = \|U_{11}^{-1} L_{11}^{-1} \|\le 
	 4^m \cdot \min\{|p_1|,\ldots,|p_{m+1}|\}^{-1}.
	\]
	This implies 
	\begin{equation} \label{eq:boundonminp}
	 \min\{|p_1|,\ldots,|p_{m+1}|\} \le 4^m \| A_{11}^{-1} \|^{-1} = 4^m \sigma_{m+1}(A_{11}) \le 4^m \sigma_{m+1}(A),
	\end{equation}
	where we used interlacing properties of singular values, see~\cite[Corollary 7.3.6]{HornJohnson2013}.
	
	On the other hand, as $A^{(m)}$ is the matrix obtained after $j$ steps of Algorithm~\ref{alg:aca} applied to the matrix $A^{(m-j)}$, the definition~\eqref{eq:growthfactor} gives the inequalities
	\[
	 p_{m+1} = \| A^{(m)} \|_{\max} \le \rho_j \cdot \|A^{(m-j)} \|_{\max} = \rho_j \cdot |p_{m-j+1}| \le \rho_m \cdot |p_{m-j+1}|
	\]
	for $j = 1,2,\ldots,m$. We therefore obtain
	\begin{eqnarray}
	 && \Vert A - A(:,J)A(I,J)^{-1} A(I,:) \Vert_{\max} \nonumber \\
	 && =\| A^{(m)} \|_{\max} = |p_{m+1}| \le \rho_m \min\{|p_1|,\ldots,|p_{m+1}|\}.\label{eq:errorIsPivot}
	\end{eqnarray}
	Combined with~\eqref{eq:boundonminp}, this shows the result of the theorem.
	  \end{proof}
	  
	  Because of the factor $4^m$, Theorem~\ref{thm:ACAgeneral} only guarantees good low-rank approximations when the singular values are strongly decaying. This limitation does not correspond to the typical behavior observed in practice; the quantities $\|{L_{11}^{-1}}\|$ and $\|{U_{11}^{-1}}\|$ rarely assume the exponential growth estimates used in the proof of Theorem~\ref{thm:ACAgeneral}. In turn, the factor $4^m$ usually severely overestimates the error. Nevertheless, there are examples for which the error estimate of Theorem~\ref{thm:ACAgeneral} is asymptotically tight; see Section~\ref{sec:acaSPSD} below. 
	  
	The matrix norms on the two sides of the estimate~\eqref{eq:aca-general} do not match. In the following we develop a variant of Theorem~\ref{thm:ACAgeneral} in which the best approximation error is also measured in terms of $\|\cdot\|_{\max}$. This will be useful later on, in Section~\ref{sec:acafunctions}, when considering approximation of functions. Let us define the approximation numbers
	\begin{equation*} 
	 \gamma_k(A):=\min\{\|E\|_{\max}:\, \text{rank}(A+E) \le k\}, \qquad k = 1,\ldots,n.
	\end{equation*}
	Because of $\|E\| / n \le \|E\|_{\max} \le \|E\|$, we have $\sigma_k(A) / n \le \gamma_{k-1}(A) \le \sigma_k(A)$. If $A$ is invertible then $\sigma_n(A) = \|A^{-1}\|^{-1}$. This result, relating the distance to singularity to the norm of the inverse, extends to general subordinate matrix norms; see, e.g.,~\cite[Theorem 6.5]{Higham2002}. In particular, we have
	\begin{equation} \label{eq:distanceinverse}
	 \gamma_{n-1}(A)=\|A^{-1}\|^{-1}_{\infty\to 1},
	\end{equation}
	with $\|\cdot \|_{\infty\to 1}$ denoting the matrix norm induced by the $1$- and $\infty$-norms. More generally, we set \[
	\|B\|_{\alpha \to \beta} :=\sup_{x\not=0} \|Bx\|_\beta / \|x\|_\alpha
	\]
	for vector norms $\|\cdot\|_\alpha$, $\|\cdot\|_\beta$.
	Note that $\|B\|_{1\to \infty}=\|B\|_{\max}$.
	
	\begin{thm}\label{thm:ACAmixednorms}
	 Under the assumptions of Theorem~\ref{thm:ACAgeneral} and with the notation introduced above, we have
	 \[
	    \Vert A - A(:,J)A(I,J)^{-1} A(I,:) \Vert_{\max} \le 2^{2m+1} \cdot \rho_m \cdot \gamma_{m}(A).  
	 \]
	\end{thm}
	\begin{proof}
	Along the lines of the proof of Theorem~\ref{thm:ACAgeneral}, we first note that 
	\[
	 \|L_{11}^{-1}\|_{\infty\to 1} \le (2^{m+1}-1) \min\{|p_1|,\ldots,|p_{m+1}|\}^{-1},
	\]
	which can be shown by induction on $m$.
	Combined with $\|U_{11}^{-1}\|_{1 \to 1} \le 2^m$, see~\cite[Theorem 6.1]{Higham1987}, we obtain
	\begin{eqnarray*}
	 \|A_{11}^{-1}\|_{\max} &=& \|A_{11}^{-1}\|_{\infty \to 1} \le \| U_{11}^{-1}\|_{1 \to 1} \|\|L_{11}^{-1}\|_{\infty\to 1} \\
	 &\le& 2^{2m+1} \min\{|p_1|,\ldots,|p_{m+1}|\}^{-1},
	\end{eqnarray*}
	where we used submultiplicativity~\cite[Eqn (6.7)]{Higham2002}. Using~\eqref{eq:distanceinverse}, this implies
	\[
	  \min\{|p_1|,\ldots,|p_{m+1}|\} \le 2^{2m+1} \| A_{11}^{-1} \|^{-1}_{\infty \to 1} = 2^{2m+1} \gamma_{m}(A_{11}) \le 2^{2m+1} \gamma_{m}(A),
	\]
	with the last inequality being a direct consequence of the definition of $\gamma_{m}$. The rest of the proof is identical with the proof of Theorem~\ref{thm:ACAgeneral}.
	\end{proof}

	\subsection{Error analysis for SPSD matrices}\label{sec:acaSPSD}
	
	In the SPSD case, the pivot elements of Algorithm~\ref{alg:aca} are always non-increasing. Thus, when restricting the supremum in~\eqref{eq:growthfactor} to SPSD matrices of rank at least $k$, one obtains $\rho_k = 1$. In turn, the following result due to Harbrecht et al.~\cite{HarbrechtPetersSchneider2012} is a corollary of Theorem~\ref{thm:ACAgeneral}.
	\begin{cor}\label{cor:ACAposdef}
	  For an SPSD matrix $A$ of rank at least $m$, the bound of Theorem~\ref{thm:ACAgeneral} improves to
	  \begin{equation*}
	    \Vert A - A(:,J)A(I,J)^{-1} A(I,:) \Vert_{\max} \le 4^m \cdot \sigma_{m+1}(A).
	  \end{equation*}
	\end{cor}
	
	The bound of Corollary~\ref{cor:ACAposdef} is asymptotically tight, see \cite[Remark 3.3]{HarbrechtPetersSchneider2012} and \cite[p. 791]{Kahan1966}.
	As the growth factor $\rho_m$ which comes into play in Theorem~\ref{thm:ACAgeneral} is small compared to the $4^m$ factor, this also proves that the bound of Theorem~\ref{thm:ACAgeneral} is almost tight. 
	

	\subsection{Error analysis for DD matrices} \label{sec:ddmatrices}
	
	When restricting the supremum in~\eqref{eq:growthfactor} to DD matrices of rank at least $k$, one obtains $\rho_k \le 2$; see Theorem 13.8 in~\cite{Higham2002}.
	\begin{cor} \label{cor:ACAdd}
	  For a DD matrix $A$ of rank at least $m$, the bound of Theorem~\ref{thm:ACAgeneral} improves to
	  \begin{equation*}
	    \Vert A - A(:,J)A(I,J)^{-1} A(I,:) \Vert_{\max} \le (m+1) \cdot 2^{m+1} \cdot \sigma_{m+1}(A).
	  \end{equation*}
	\end{cor}
	
	\begin{proof}
	It is well known that the factor $U$ in the LU decomposition of a DD matrix  is again DD; see~\cite{CarlsonMarkham1979}. In particular, this implies that the $(m+1)\times (m+1)$ unit upper triangular matrix $U_{11}$ in the proof of Theorem \ref{thm:ACAgeneral} is DD. Then, for every entry of $U_{11}^{-1}$ we have $ |(U_{11}^{-1})_{ij}|  \le 1$ by \cite[Prop. 2.1]{Pena2004}. Therefore,
	  \begin{equation}\label{eq:invU}
	    \Vert U_{11}^{-1} \Vert \le  \Vert U_{11}^{-1} \Vert_F \le \sqrt{(m+1)(m+2)/2} 
	\le m+1.
	\end{equation}
	This shows that the factor $4^m$ can be reduced to $(m+1)2^m$ in the bound of Theorem \ref{thm:ACAgeneral}. Combined with $\rho_m \le 2$, this establishes the desired result.
	\end{proof}

	\begin{cor}\label{cor:ACAddd}
	  For a doubly DD matrix $A$ of rank at least $m$, the bound of Corollary~\ref{cor:ACAdd} improves to
	  \begin{equation*}
	    \Vert A - A(:,J)A(I,J)^{-1} A(I,:) \Vert_{\max} \le 2 \cdot (m+1)^2 \cdot \sigma_{m+1}(A).
	  \end{equation*}
	\end{cor}
	
	\begin{proof}
	    Trivially, $A^*$ is DD. By the same arguments as in the proof of Corollary~\ref{cor:ACAdd} this implies that not only $U_{11}$ but also $L_{11}^*$ is DD. Proceeding as in the derivation of~\eqref{eq:invU}, we get
	    \begin{equation*}
	        \| L_{11}^{-1} \| \le \| L_{11}^{-1} \|_F \le (m+1) \cdot \min \{ |p_1|, \ldots, |p_{m+1}| \}^{-1}.
	    \end{equation*}
	    This shows that the factor $(m+1) \cdot 2^{m+1}$ of Corollary~\ref{cor:ACAdd} can be improved to $2(m+1)^2$.
	\end{proof}
	
	\begin{remark} \label{remark:polyapprox}
	The fact that Algorithm~\ref{alg:aca} gives a polynomially good low-rank approximation of a doubly DD matrix does not imply that it also gives a polynomially good approximation of the maximum volume submatrix. For instance, let $n = 2m$ and consider $A = \begin{bmatrix} I_m & 0 \\ 0 & B_m \end{bmatrix}$, where $B_m = \tridiag[\frac{1}{2}, 1, -\frac{1}{2}]$.
	Then Algorithm~\ref{alg:aca} does not perform any pivoting during its $m$ steps and thus the submatrix $I_m$ is selected. Its volume is $1$, while the volume of $B_m$ is exponentially larger, it grows like $\left ( \frac{1+\sqrt{2}}{2} \right ) ^m$.
\end{remark}
	
\subsection{Tightness of estimates for DD matrices}

To study the tightness of the estimates from Section~\ref{sec:ddmatrices}, it is useful to connect Algorithm~\ref{alg:aca} to LDU decompositions. Suppose that the application of $m = n-1$ steps of Algorithm~\ref{alg:aca} yields $I = J = \{1,\ldots,n-1\}$. As in the proof of Theorem~\ref{thm:ACAgeneral}, we exploit the relation~\eqref{eq:sumlu} for $k = m+1 = n$ to obtain the factorization
\[
 A = L_{11} U_{11} = LDU,\quad D:= \text{diag}(p_1, \ldots, p_{n})
\]
where $L:=L_{11} D^{-1}$ and $U := U_{11}$ are lower and upper unit triangular matrices, respectively.
We recall from~\eqref{eq:errorIsPivot} that the error of the approximation returned by Algorithm~\ref{alg:aca} is governed by $|p_{n}|$. 

From now on, let $A$ be a DD matrix. In this case, the pivot growth factor does not exceed $2$ and we have that $|p_{n}| \le 2\|D^{-1}\|^{-1} \le 2|p_{n}|$. In turn, the ratio between $|p_{n}|$ and the best rank-$(n-1)$ approximation error satisfies
\begin{equation} \label{eq:defrm}
 r_m := \frac{|p_n|}{\sigma_{n}(A)} = |p_{n}|\, \| A^{-1} \| \le |p_{n}| \, \| U^{-1} \|\, \| D^{-1} \|\, \| L^{-1} \| \le 2 \| L^{-1} \|\, \|U^{-1}\|.
\end{equation}
Inheriting the diagonal dominance from $A$, the matrix $U$ is well conditioned; see~\eqref{eq:invU}. Therefore, large $r_m$ require $\|L^{-1}\|$ to become large.
	
The quantity $\|L^{-1}\|$ also plays a prominent role in the stability analysis of LDU decompositions, see~\cite{DopicoKoev2011} and the references therein. In particular, the \emph{potential} rapid growth of $\|L^{-1}\|$ under complete pivoting has motivated the search for alternative pivoting strategies~\cite{Pena2004}. However, the existing literature is scarce on examples actually exhibiting such rapid growth. The worst example we could find is by Barreras and Peña~\cite[Sec. 3]{BarrerasPena2012}, which exhibits linear growth. A more rapid growth is attained by the $n \times n$ matrix
\begin{equation*}
	    A = \left [ {
	    \begin{array}{cccc|cccc}
	      1 & -1 & & & & & & \\
	      & 1 & \ddots & & & & & \\
	      & & \ddots & -1 & & & & \\
	      & & & 1 & -\frac{1}{n/2+1} & -\frac{1}{n/2+1} & \cdots & -\frac{1}{n/2+1} \\
	      \hline 
	      -1 & & & & 1 & & & \\
	      -1 & & & & & 1 & & \\
	      \vdots & & & & & & \ddots & \\
	      -1 & & & & & & & 1 \\
	    \end{array}
	    } \right ],
\end{equation*}
where $n$ is even and each block has size $n/2 \times n/2$. When applying complete pivoting to this matrix, no interchanges are performed and the LDU factorization satisfies
\[ \| L^{-1} \| = \Theta(m\sqrt{m}),\quad \| D^{-1} \| = 1/|p_n| = 2, \quad \| U^{-1} \| = \Theta(m).\]
Note that, for this example, the right-hand side of~\eqref{eq:defrm} overestimates the error. This example attains quadratic growth: $r_m = \| A^{-1} \| = \Theta(m^2)$. This is still far away from the exponential growth estimated in Corollary~\ref{cor:ACAdd}, but closer than the example from~\cite[Sec. 3]{BarrerasPena2012}, which yields $r_m = \Theta(m\sqrt{m})$.
	
	For a doubly DD matrix, one obtains linear growth in~\eqref{eq:defrm} by considering the $n \times n$ lower bidiagonal matrix $B$ having $1$ on the diagonal and $-1$ on the first subdiagonal. In this case, $L = B$, $D = U = I_n$ and hence $\| B_n^{-1} \| = \Theta(m)$, showing that $r_m$ can grow at least linearly with $m$. We have not found an example exhibiting the quadratic growth estimated by Corollary~\ref{cor:ACAddd}.

	\subsection{Cross approximation for functions} \label{sec:acafunctions}
	Let us consider the approximation of a function $f:[-1,1]^2\to\R$ by a sum of separable functions:
	\begin{equation*}
	    f(x, y) \approx \sum_{i = 1}^M f_i^{(1)}(x) \cdot f_i^{(2)}(y).
	\end{equation*}
	In the context of \emph{cross approximation}, the factors are restricted to functions $f_i^{(1)}$ of the form $f_i^{(1)} = f(x, \bar y_i)$ and $f_i^{(2)} = f(\bar x_i, y)$, where $\bar x_i$ and $\bar y_i$ are fixed elements of $[-1,1]$. 
	In particular, Micchelli and Pinkus~\cite{MicchelliPinkus1978} considered interpolating approximations of the following form:
	\begin{equation*}
	    f(x,y) \approx \begin{bmatrix} f(x, y_1) \\ \cdots \\ f(x, y_m) \end{bmatrix}^* \cdot  \begin{bmatrix} f(x_1,y_1) & \cdots & f(x_1, y_m)  \\
	\vdots & & \vdots  \\
	f(x_m, y_1) & \cdots & f(x_m, y_m) 
	\end{bmatrix}^{-1}  \cdot \begin{bmatrix} f(x_1, y) \\ \vdots \\ f(x_m, y) \end{bmatrix},
	\end{equation*}
	for some $x_1, \ldots, x_m, y_1, \ldots, y_m \in [-1,1]$. Townsend and Trefethen~\cite{TownsendTrefethen2015} use a strategy for choosing the interpolation points which is basically equivalent to Algorithm~\ref{alg:aca} and they prove a convergence result under some analyticity hypotheses on the function $f$.
	There also exist error analyses for cross approximation of functions when using different pivoting strategies, see, e.g.,~\cite{Bebendorf2011, Schneider2010}.
	
	Algorithm 2 summarizes cross approximation of functions with complete pivoting.
	
	\begin{algorithm}
		\caption{Cross approximation of functions~\cite[Figure 2.1]{Townsend2014} \label{alg:acaf}}
		\begin{algorithmic}[1]
			\REQUIRE{$f:[-1,1]^2\to\R$ and $m > 0$}
			\STATE{$e_0(x,y) = f(x,y)$}
			\STATE{$f_0(x,y) = 0$}
			\STATE{$k = 0$}
			\FOR{$k = 1, \ldots, m$}
			\STATE{$(x_{k+1},y_{k+1}) := \argmax_{(x,y) \in [-1,1]^2} \{ | e_k(x,y) | \}$}
			\STATE{$e_{k+1} := e_k - \frac{e_k(x_{k+1},\cdot) \cdot e_k(\cdot, y_{k + 1})}{e_k(x_{k+1}, y_{k+1})}$}
			\STATE{$f_{k+1} := f_k + \frac{e_k(x_{k+1}, \cdot) \cdot e_k(\cdot, y_{k + 1})}{e_k(x_{k+1}, y_{k+1})}$}
			\ENDFOR
		\end{algorithmic}
	\end{algorithm}

	We now explain the connection to Algorithm~\ref{alg:aca}.
	Fix $(x,y) \in [-1,1]^2$ and consider the points $x_1, \ldots, x_m$ and $ y_1, \ldots, y_m$ obtained by the first $m$ steps of Algorithm~\ref{alg:acaf}. Consider what happens when applying Algorithm~\ref{alg:aca} to the $(m+1)\times(m+1)$ matrix obtained by interpolating $f$ in the points mentioned above:
	\begin{equation*}
	A_{(x,y)} := \begin{bmatrix}
	f(x_1,y_1) & \cdots & f(x_1, y_m) & f(x_1, y) \\
	\vdots & & \vdots & \vdots \\
	f(x_m, y_1) & \cdots & f(x_m, y_m) & f(x_m, y) \\
	f(x, y_1) & \cdots & f(x, y_m) & f(x,y)
	\end{bmatrix}.
	\end{equation*}
	The first chosen pivot will be $p_1 = f(x_1, y_1)$ because it is the largest entry of the matrix. Now observe that the Schur complement $A^{(1)}$ obtained after the first step, is the matrix that interpolates the function $e_1$ in the points $x_2, \ldots, x_m, x$ and $y_2, \ldots, y_m, y$. At this point, the second pivot chosen by Algorithm~\ref{alg:aca} will be $e_1(x_2,y_2)$ because of how Algorithm~\ref{alg:acaf} chose $(x_2,y_2)$ in line $5$. After $m$ steps of Algorithm~\ref{alg:aca} we will be left with only one nonzero entry in position $(m+1,m+1)$ and this will be $e_m(x,y)$.
	This allows us to estimate $|e_m(x,y)|$ via Theorem~\ref{thm:ACAmixednorms}:
	\begin{equation}\label{eq:errorfunction}
	|e_m(x,y)| \le  2^{2m+1} \cdot \rho_m \cdot\gamma_m(A_{(x,y)}).
	\end{equation}
	The last thing we need is an estimate on $\gamma_m(A_{(x,y)})$ that is uniform in $(x,y) \in [-1,1]^2$. This will follow from analyticity assumptions on the functions $f(\cdot, y)$ for $y \in[-1,1]$.
	\begin{defn}
		The Bernstein ellipse $\mathcal{E}_r$ of radius $r > 1$ is the ellipse with foci in $-1$ and $1$ and with sum of the semi-axes equal to $r$.
	\end{defn}
	
	\begin{cor}\label{cor:acaf2}
		Let $f: [-1,1]^2 \to \R$ be such that $f(\cdot, y)$ admits an analytic extension - which we will denote by $\tilde f$ - in the Bernstein ellipse $\mathcal{E}_{r_0}$ of radius $r_0$ for each $y \in [-1, 1]$. Let $1 < r < r_0$ and 
		\begin{equation*}
		M := \sup_{\eta \in \partial \mathcal{E}_r, \, \xi \in [-1,1]} |\tilde f(\eta, \xi)|.
		\end{equation*}
		After $m$ steps of Algorithm~\ref{alg:acaf} the error function satisfies
		\begin{equation*}
		\| e_m \|_{\max} \le \frac{2M\rho_m}{1-1/r} \cdot \left ( \frac{r}{4} \right )^{-m}.
		\end{equation*}
	\end{cor}
	\begin{proof}
	    Fix $(x,y) \in [-1,1]^2$ and let $b: [-1,1] \to \R^{m+1}$ be the vector-valued function defined by
	    \begin{equation*}
	        b(\eta) := \begin{bmatrix} f(\eta, y_1) & \cdots & f(\eta, y_m) & f(\eta, y) \end{bmatrix}^*.
	    \end{equation*}
	    The analyticity hypothesis allows us to apply standard polynomial approximation results (see e.g. Corollary 2.2 in~\cite{KressnerTobler2011}) and conclude that there exists an approximation $\hat b: [-1,1] \to \R^{m+1}$ given by
	    \begin{equation}\label{eq:bhat}
	        \hat b (\eta) = \sum_{k=1}^m p_k(\eta) v_k,
	    \end{equation}
	    where $v_k \in \R^{m+1}$ are constant vectors and $p_k : [-1,1] \to \R$ are polynomials, such that
	    \begin{equation*}
	        \max \| b(\eta) - \hat b(\eta) \|_{\max} \le \frac{2}{1-r^{-1}} \cdot \max_{\alpha \in \mathcal{E}_r} \| b(\alpha) \|_{\max} \cdot r^{-m}
	    \end{equation*}
	    for any $1 < r < r_0$. We can clearly bound $\max_{\alpha \in \mathcal{E}_r} \| b(\alpha) \|_{\max} \le M$.
	
	    The matrix $A_{(x,y)}$ is obtained by sampling $b$ in the points $x_1, \ldots, x_m, x$, i.e.
	    \begin{equation*}
	        A_{(x,y)} = \begin{bmatrix} b(x_1) & \cdots & b(x_m) & b(x) \end{bmatrix}.
	    \end{equation*}
	    Let us define, analogously,
	    \begin{equation*}
	        \hat A_{(x,y)} = \begin{bmatrix} \hat b(x_1) & \cdots & \hat b(x_m) & \hat b(x) \end{bmatrix}.
	    \end{equation*}
	    Notice that $\hat A_{(x,y)}$ has rank as most $m$ because by~\eqref{eq:bhat} each of the $m+1$ columns of $\hat A_{(x,y)}$ is a linear combination of the $m$ vectors $v_1, \ldots, v_m$, so
	    \begin{align*}
	        \gamma_m(A_{(x,y)}) & \le \| A_{(x,y)} - \hat A_{(x,y)} \|_{1\to\infty} = \max_{\alpha \in \{x_1, \ldots, x_m, x\}} \| b(\alpha) - \hat b(\alpha) \|_{\max} \\
	        & \le \frac{2M}{1-r^{-1}} \cdot r^{-m}.
	    \end{align*}
		The result then follows from Equation~\eqref{eq:errorfunction}.
	\end{proof}

		To get convergence of the error function to zero as $m\to\infty$, in Corollary~\ref{cor:acaf2} it is sufficient that the function $f(\cdot, y)$ admits an analytic extension to the Bernstein ellipse $\mathcal{E}_{r_0}$ with $r_0 > 4$ for each $y$, because the factor $\rho_m$ has subexponential growth. Our result compares favorably to Theorem 8.1 in~\cite{TownsendTrefethen2015}, which requires an analytic extension to the region $K$ consisting of all points at a distance $\le 4$ from $[-1,1]$.
		
	Figure \ref{fig:convergenceregions} compares the two domains and it is evident that the requirement from~\cite{TownsendTrefethen2015} is significantly more restrictive.
		
		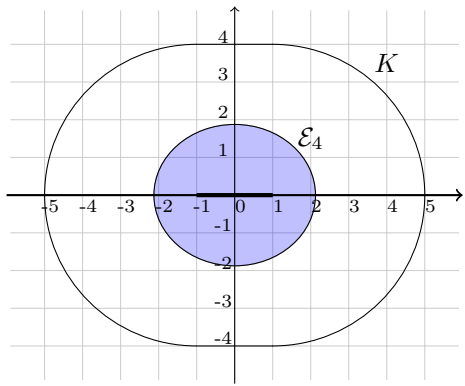
\begin{figure}
			\label{fig:convergenceregions}
			
			\begin{center}
				\begin{tikzpicture}[scale=0.5]

				\draw[very thin, gray!40!white] (-5.9, -4.9) grid (5.9, 4.9);
				\draw[fill = blue, opacity = 0.25] (0,0) ellipse (2.125cm and 1.875cm);
				\draw(0,0) ellipse (2.125cm and 1.875cm);
				\draw[thick] (-1,0) -- (1,0);
				\begin{scope}
				\clip (-6,-5) rectangle (-1,5);
				\draw (-1,0) circle(4);
				\end{scope}
				\begin{scope}
				\clip (6,-5) rectangle (1,5);
				\draw (1,0) circle(4);
				\end{scope}
				\draw (-1,4) -- (1,4);
				\draw (-1,-4) -- (1,-4);
				\draw[ultra thick] (-1,0) -- (1,0);
				
				\draw[->] (0,-5) -- (0,5);
				\draw[->, thick] (-6,0) -- (6,0);
				
				\node at (2,1.5) {$\mathcal{E}_4$};
				\node at (4,3.5) {$K$};
				
				\foreach \x in {-5,-4,-3,-2,-1,0,1,2,3,4,5}
				\node at (\x +0.15, -0.3) {\scriptsize{\x}};
				\foreach \y in {-4,-3,-2,-1,1,2,3,4}
				\node at (-0.3,\y + 0.2) {\scriptsize{\y}};
				\end{tikzpicture}
			\end{center}
			\caption{Analyticity regions ensuring convergence of Algorithm~\ref{alg:acaf} according to Corollary~\ref{cor:acaf2} and~\cite[Theorem 8.1]{TownsendTrefethen2015}.}
		\end{figure}
		
		For a positive semidefinite kernel function $f$, the matrix $A_{(x,y)}$ in~\eqref{eq:errorfunction} is positive semidefinite and hence the bound of Corollary~\ref{cor:acaf2} holds with $\rho_m\equiv 1$. This matches an asymptotic result given in~\cite[Theorem 9.1]{TownsendTrefethen2015}.

	\section{Conclusions}
	
	The fact that the search for the maximum volume submatrix can be restricted to principal submatrices for SPSD and DD matrices appears intuitive and is sometimes used without proof, see, e.g., \cite[Theorem 1]{Nikolov2015}. As far as we know, Theorems~\ref{thm:spsd} and~\ref{thm:maxvolDD} are the first results providing a mathematical justification to this intuition.
	
	For cross approximation, Theorem~\ref{thm:ACAgeneral} appears to be the first non-asymptotic error bound that holds for general matrices. Except for~\cite{HarbrechtPetersSchneider2012},  previous results for cross approximation applied to matrices or functions~\cite{Bebendorf2011,TownsendTrefethen2015} are based on a step-by-step analysis of the error. In contrast, our technique takes a more global view and can, in turn, leverage existing results on the pivot growth in Gaussian elimination. As illustrated in Section~\ref{sec:acafunctions}, this can yield significant advantages.
	
	A number of fundamental questions remain open. Most importantly, there is a mismatch between the derived error bounds and the known worst-case examples for cross approximation applied to DD and doubly DD matrices. Especially for DD matrices, this problem appears to be difficult to overcome and was encountered previously in the context of the error analysis of LDU factorizations~\cite{BarrerasPena2012,DopicoKoev2011}.
	
	\section{Acknowledgments}
	
	We thank Froilan Dopico for an insightful discussion on DD matrices. 
	
	\bibliographystyle{abbrv}
	\bibliography{Bib} 
	
	\end{document}